\documentclass{amsart} 
\usepackage{microtype} 
\pdfoutput=1

\newcommand{\secone}{\section}
\newcommand{\sectwo}{\subsection}
\newcommand{\secthree}{\subsubsection}

\usepackage{amsmath}
\usepackage{amsthm}
\usepackage{amssymb}

\usepackage{tikz}
\usepackage{tikz-cd} 
\tikzset{/tikz/commutative diagrams/row sep/normal=1.8em} 
\usepackage{caption}
\usepackage{subcaption}

\usetikzlibrary{patterns}
\usepackage[all, pdf]{xy}
\usepackage{enumerate}
\usepackage{graphicx}
\usepackage{color}
\usepackage{transparent}
\usepackage{url}

\usepackage{hyphenat}

\usepackage[breaklinks]{hyperref} 
\usepackage{xcolor}
\definecolor{dark-red}{rgb}{0.4,0.15,0.15}
\definecolor{dark-blue}{rgb}{0.15,0.15,0.4}
\definecolor{medium-blue}{rgb}{0,0,0.5}
\hypersetup{
    colorlinks, linkcolor={dark-red},
        citecolor={dark-blue}, urlcolor={medium-blue}
}

\numberwithin{equation}{section}
\newtheorem{theorem}{Theorem}[section]
\newtheorem{proposition}[theorem]{Proposition}
\newtheorem{lemma}[theorem]{Lemma}
\newtheorem{corollary}[theorem]{Corollary}

\newtheorem*{fact*}{Fact}
\theoremstyle{definition}
\newtheorem{remark}[theorem]{Remark}

\DeclareMathOperator{\Aut}{Aut}

\newcommand{\toiso}{\stackrel{\sim}{\longrightarrow}}

\newcommand{\QQ}{\mathbb{Q}}
\newcommand{\ZZ}{\mathbb{Z}}
\newcommand{\RR}{\mathbb{R}}
\newcommand{\NN}{\mathbb{N}}

\newcommand{\gkappa}[1]{\kappa_{#1}} 

\newcommand{\Diff}{\operatorname{Diff}}

\newcommand{\Hirz}{\tilde{\mathcal{L}}}
\newcommand{\mpoint}{\star} 

\newcommand{\gpsi}[2]{{#1}_{\left(#2\right)}} 
\newcommand{\intclass}[1]{\nu_{\left(#1\right)}}  

\title[Tautological rings]{Tautological rings\\for high dimensional manifolds}
\author{S{\o}ren Galatius}
\email{galatius@stanford.edu}
\address{Department of Mathematics\\
	Stanford University\\
	Stanford CA, 94305}

\author{Ilya Grigoriev}
\email{ilyagr@math.uchicago.edu }
\address{
Department of Mathematics\\
University of Chicago\\
5734 S.\ University Avenue\\
Chicago, IL 60637
}

\author{Oscar Randal-Williams}
\email{o.randal-williams@dpmms.cam.ac.uk}
\address{Centre for Mathematical Sciences\\
Wilberforce Road\\
Cambridge CB3 0WB\\
UK}

\begin{document}
\begin{abstract}
We study tautological rings for high dimensional manifolds, that is, for each smooth manifold $M$ the ring $R^*(M)$ of those of characteristic classes of smooth fibre bundles with fibre $M$ which is generated by generalised Miller--Morita--Mumford classes. We completely describe these rings modulo nilpotent elements, when $M$ is a connected sum of copies of $S^n \times S^n$ for $n$ odd.
\end{abstract}

\maketitle

\secone{Introduction}

Let $E^{k+d}$ and $B^k$ be connected compact smooth oriented manifolds, and $\pi : E \to B$ be a smooth fibre bundle (i.e.\ a proper submersion) with fibre $M^{d}$. Then there is a $d$-dimensional oriented vector bundle $T_\pi = \mathrm{Ker}(D\pi)$ over $E$. If $c \in H^*(BSO(d);\QQ)$ is a rational characteristic class of such vector bundles then we define the associated \emph{generalised Miller--Morita--Mumford class} to be the fibre integral
$$\kappa_c(\pi) = \int_\pi c(T_\pi ) \in H^{*-d}(B;\QQ).$$

For a basis $\mathcal{B} \subset H^*(BSO(d);\QQ)$ this defines a ring homomorphism
\begin{align*}
\QQ[\kappa_c \, \vert \, c \in \mathcal{B}] &\longrightarrow H^*(B;\QQ)\\
\kappa_c &\longmapsto \kappa_c(\pi).
\end{align*}
We let $I_M \subset \QQ[\kappa_c \, \vert \, c \in \mathcal{B}]$ be the ideal consisting of those polynomials in the classes $\kappa_c$ which vanish on every such smooth fibre bundle with fibre $M$, and define the \emph{tautological ring} of $M$ to be the associated quotient ring
$$R^*(M) =  \QQ[\kappa_c \, \vert \, c \in \mathcal{B}]/I_M.$$
The name ``tautological ring'' is borrowed from the case $d=2$, where it usually refers to the subring of the cohomology ring (or Chow ring) of moduli spaces of Riemann surfaces which is generated by certain tautological classes $\kappa_i$ (which in our notation correspond to $\kappa_{e^{i+1}}$). Our definition coincides with the usual one in this case. There is a large literature on these rings; see \cite{Mumford, Looijenga, Faber, Morita}.

We shall be interested in manifolds of even dimension $d=2n$, in which case $H^*(BSO(2n);\QQ) = \QQ[p_1,p_2,\ldots, p_{n-1},e]$ and we take $\mathcal{B}$ to be the basis of monomials in these classes. Writing $W_g$ for the connected sum of $g$ copies of $ S^n \times S^n$, our goal in this paper is to study the structure of $R^*(W_g)$ modulo the nilradical $\sqrt{0}$, i.e.\ the ideal of nilpotent elements, and our main result can be stated as follows.

\begin{theorem}\label{thm:Main1}
Let $n$ be odd. Then
\begin{enumerate}[(i)]
\item\label{it:Main1:1} $R^*(W_0)/\sqrt{0} = \QQ[\kappa_{ep_1}, \kappa_{ep_2}, \ldots, \kappa_{ep_{n}}]$,

\item\label{it:Main1:2} $R^*(W_1)/\sqrt{0} = \QQ$,

\item\label{it:Main1:3} $R^*(W_g)/\sqrt{0} = \QQ[\kappa_{ep_1}, \kappa_{ep_2}, \ldots, \kappa_{ep_{n-1}}]$ for $g > 1$.
\end{enumerate}
\end{theorem}
In fact, (\ref{it:Main1:1}) already holds before taking the quotient by the nilradical, and holds for $n$ both even and odd (see Section~\ref{sec:genus-zero}).

As the Krull dimension of a ring is unchanged by taking the quotient by the nilradical, we may also conclude that $R^*(W_0)$ has Krull dimension $n$, $R^*(W_1)$ has Krull dimension 0, and $R^*(W_g)$ has Krull dimension $n-1$ as long as $g > 1$.

We shall also study two closely related tautological rings. Firstly, we may consider fibre bundles $\pi : E^{k+2n} \to B^k$ with fibre $M^{2n}$ equipped with a section $s : B \to E$. To such a fibre bundle and a $c \in \mathcal{B}$ we may also associate the characteristic class $c(\pi, s) = s^*(c(T_\pi))$, so in this case there is defined a ring homomorphism
\begin{align*}
\QQ[c, \kappa_c \, \vert \, c \in \mathcal{B}] &\longrightarrow H^*(B;\QQ)\\
c & \longmapsto c(\pi, s)\\
\kappa_c &\longmapsto \kappa_c(\pi)
\end{align*}
and we let $I_{(M,\mpoint)} \subset \QQ[c, \kappa_c \, \vert \, c \in \mathcal{B}]$ denote the ideal of those polynomials in the $c$ and $\kappa_c$ which vanish on all such bundles; we write $R^*(M,\mpoint)$ for the associated quotient ring.

\begin{theorem}\label{thm:Main2}
Let $n$ be odd. 
\begin{enumerate}[(i)]
\item\label{it:Main2:1} For each $c \in \mathcal{B}$, $(2-2g)\cdot c = \kappa_{e c} \in R^*(W_g,\mpoint)/\sqrt{0}$. 

\item\label{it:Main2:2} For $g \neq 1$, the map
$$R^*(W_g)/\sqrt{0} \longrightarrow R^*(W_g,\mpoint)/\sqrt{0}$$
is an isomorphism.

\item\label{it:Main2:3} The class $e \in R^*(W_1, \mpoint)$ is nilpotent, and the map
$$\QQ[p_1, p_2, \ldots, p_{n-1}] \longrightarrow R^*(W_1, \mpoint)/\sqrt{0}$$
is an isomorphism.
\end{enumerate}
\end{theorem}

Secondly, we may consider fibre bundles $\pi : E^{k+2n} \to B^k$ with fibre $M^{2n}$ equipped with an embedding $S : D^{2n} \times B \to E$ over $B$ (or equivalently, with a normally framed section). The restriction $s = S\vert_{\{0\} \times B}$ defines a section, but the fact that it is framed means that $c(\pi,s)=0$. Thus in this case we denote by $I_{(M,D^{2n})} \subset \QQ[\kappa_c \, \vert \, c \in \mathcal{B}]$ the ideal of polynomials in the $\kappa_c$ which vanish on all such bundles, and write $R^*(M,D^{2n})$ for the associated quotient ring. The following is immediate from Theorem \ref{thm:Main1} and Theorem \ref{thm:Main2} (\ref{it:Main2:1}).

\begin{corollary}\label{cor:Main3}
If $n$ is odd then $R^*(W_g, D^{2n})/\sqrt{0} = \QQ$ for all $g$.
\end{corollary}

The second-named author has shown~\cite[Theorem 1.1]{grigoriev-relations} that for $n$ odd the ring $R^*(W_g)$ is finitely-generated as a $\QQ$-algebra, and hence $R^*(W_g, D^{2n})$ is too. (In fact, Theorem 1.1 of \cite{grigoriev-relations} only shows this for $g > 1$, but in Proposition \ref{prop:fg} we will show that $R^*(W_g,\mpoint)$, and hence $R^*(W_g, D^{2n})$, is finitely-generated for all $g$.) Corollary \ref{cor:Main3} therefore implies the following.

\begin{corollary}\label{cor:Main4}
If $n$ is odd then $R^*(W_g, D^{2n})$ is a finite-dimensional rational vector space for all $g$.
\end{corollary}

It would be interesting to have a computation of the ring $R^*(W_g, D^{2n})$ for some $g > 1$. In proving these theorems, we obtain results about the vanishing of certain elements in $R^*(M)/\sqrt{0}$ for any manifold $M$ (in Theorem~\ref{thm:hirz-nilpotent}), and about the algebraic independence of certain elements in $R^*(W_g)/\sqrt{0}$ for $n$ even (in Theorem~\ref{thm:algebraic-independence}). We cannot obtain results as conclusive as Theorem \ref{thm:Main1} for $n$ even, as our argument relies on \cite{grigoriev-relations} which does not apply in this case.

\sectwo{Acknowledgments}
This project was initiated while the three authors were at the ``Algebraic topology'' program at MSRI, supported by NSF Grant No.\ 0932078000, and continued while the authors were at the ``Homotopy theory, manifolds, and field theories'' program at the Hausdorff Research Institute for Mathematics. SG was partially supported by NSF grants DMS-1105058 and DMS-1405001. ORW was partially supported by EPSRC grant EP/M027783/1.

\secone{Nilpotence of the Hirzebruch $\mathcal{L}$-classes}\label{sec:Hirz}

In this section, we prove that certain generalised Miller--Morita--Mumford classes are nilpotent in $R^*(M)$ for \emph{any} smooth even-dimensional manifold $M^{2n}$. We will later apply this together with results from~\cite{grigoriev-relations} to the manifolds $W_g$ in order to obtain an upper bound on $R^*(W_g)/\sqrt{0}$.

  \sectwo{(Modified) Hirzebruch $\mathcal{L}$-classes}\label{sec:hirzebruch-classes}
  We first define certain cohomology classes $\Hirz_i \in H^{4i}(BSO(2n);\QQ)$. Consider the graded ring $\QQ[x_1, \ldots, x_n]$ in which each variable $x_1, \ldots, x_n$ has degree 2.  We define homogenous symmetric polynomials $\Hirz_i$ by the expression
\[ \Hirz = 2^n + \Hirz_1 + \Hirz_2 + \cdots = \prod_{i=1}^{n} \frac{x_i}{\operatorname{tanh} x_i/2}.\]
(The function $\frac{x}{\operatorname{tanh} x/2}$ is even, so the above expresses $\Hirz_i$ as a symmetric polynomial in $x_1^2, x_2^2, \ldots, x_n^2$.) The subring of symmetric polynomials is generated by the elementary symmetric polynomials $\sigma_{i,n}(x_1^2, \ldots, x_n^2)$, so we may express $\Hirz_i$ as a polynomial $\Hirz_i(\sigma_{1,n}, \ldots, \sigma_{n,n})$, and we write
$$\Hirz_i = \Hirz_i(p_1, p_2, \ldots, p_n) \in H^{4i}(BSO(2n);\QQ)$$
for this polynomial evaluated at the Pontrjagin classes. Note that these differ from the usual Hirzebruch $L$-classes, $\mathcal{L}_i$, by $\Hirz_i = 2^{n-2i} \mathcal{L}_i$. These classes may be written as
$$\Hirz_i = 2^{n}(2^{2i-1}-1) \frac{B_i}{(2i)!} \cdot p_i + (\text{polynomial in lower Pontrjagin classes})$$
where $B_i$ is the $i$th Bernoulli number, cf.\ \cite[Problem~19-C]{milnostash-characteristic} (this uses the convention of \cite[Appendix B]{milnostash-characteristic}, in which $B_i = (-1)^i \cdot 2i \cdot \zeta(1-2i)$). In particular, these leading coefficients are never zero, so $\Hirz_1, \Hirz_2, \ldots, \Hirz_n$ generate the same subring of $H^*(BSO(2n);\QQ)$ as $p_1, p_2, \ldots, p_n$.

\sectwo{Nilpotence due to the Hirzebruch signature thorem}

For any manifold $M^{2n}$ the characteristic classes $\Hirz_i \in H^{4i}(BSO(2n);\QQ)$ give rise to the corresponding generalised Miller--Morita--Mumford classes $\gkappa{\Hirz_i} \in R^*(M)$. Our present goal is to prove the following.

\begin{theorem}\label{thm:hirz-nilpotent}
The classes $\gkappa{\Hirz_i}\in R^*(M)$ are nilpotent for all natural numbers $i\geq 1$ such that $4i - 2n \neq 0$ (all natural numbers if $n$ odd).
\end{theorem}

  The main ingredient in the proof is a parametrised version of the Hirzebruch signature theorem, due to Atiyah \cite{atiyah-signature}, which relates the classes $\gkappa{\Hirz_i}$ with classes arising from the cohomology of an arithmetic group. 

Let $M^{2n}$ be an oriented manifold, and let $H = H^n(M;\ZZ)/\mathrm{torsion}$. It is a consequence of Poincar{\'e} duality that the $(-1)^n$-symmetric pairing
$$\lambda : H \otimes_\ZZ H \stackrel{\cup}{\longrightarrow} H^{2n}(M;\ZZ)/\text{torsion} = H^{2n}(M;\ZZ) \toiso \ZZ$$
is non-degenerate. Any homotopy equivalence of $M$ preserving its orientation therefore induces an isometry of the $(-1)^n$-symmetric form $(H, \lambda)$. In particular if $\pi : E \to B$ is a smooh oriented fibre bundle with fibre $M$, there is a homomorphism
$$\phi: \pi_1(B) \longrightarrow \mathrm{Aut}(H, \lambda)$$
from the fundamental group of $B$ to the group of isometries of $(H,\lambda)$, inducing a map $\phi : B \to B\mathrm{Aut}(H, \lambda)$.

We may further consider the induced $(-1)^n$-symmetric form $\lambda_\RR$ on $H_\RR = H \otimes_\ZZ \RR$. There are now two cases, depending on the parity of $n$. If $n$ is odd then $(H_\RR, \lambda_\RR)$ is a non-degenerate  skew-symmetric form over $\RR$, and so is determined by its rank. This identifies $\Aut(H_\RR, \lambda_\RR) = \mathrm{Sp}_{2g}(\RR)$, where $H_\RR$ has dimension $2g$. If $n$ is even then $(H_\RR, \lambda_\RR)$ is a non-degenerate symmetric form over $\RR$, and so is determined by its rank and signature. This identifies $\Aut(H_\RR, \lambda_\RR) = \mathrm{O}_{p,q}(\RR)$, where $H_\RR$ has dimension $p+q$ and signature $p-q$. We obtain a composition
\begin{equation*}
B \overset{\phi}\longrightarrow B\mathrm{Aut}(H, \lambda) \longrightarrow B\mathrm{Aut}(H_\RR, \lambda_\RR) = 
\begin{cases}
B\mathrm{Sp}_{2g}(\RR)\\
B\mathrm{O}_{p,q}(\RR).
\end{cases}  
\end{equation*}
Now Atiyah's work \cite[Section 4]{atiyah-signature} implies that the classes $\kappa_{\Hirz_i}$ are pulled back via this composition, and so in particular they are in the image of $\phi^* : H^*(B\mathrm{Aut}(H, \lambda);\QQ) \to H^*(B;\QQ)$. Theorem \ref{thm:hirz-nilpotent} is then an immediate consequence of the fact that all positive degree elements of $H^*(B\mathrm{Aut}(H, \lambda);\QQ)$ are nilpotent: this is because $\mathrm{Aut}(H, \lambda)$ is an arithmetic group, so has finite virtual cohomological dimension and so in particular finite $\QQ$-cohomological dimension.

\secone{Relations modulo nilpotence}\label{sec:Rels}

The main result (Theorem 2.6) of~\cite{grigoriev-relations} implies that, when $n$ is odd, for any oriented manifold bundle $E \to B$ with fibre $W_g$ and even-dimensional classes $a,b \in H^*(E;\QQ)$ such that $\pi_!(a) = 0$ and $\pi_!(b)=0$, the class $\pi_!(ab)$ is nilpotent. We shall use the following stronger statement, which will be deduced quite formally from \cite[Theorem 2.6]{grigoriev-relations}.

\begin{theorem}\label{thm:nilpotence-for-products}
Let $\pi : E \to B$ be a fibration with homotopy fibre homotopy equivalent to $W_g =g(S^n \times S^n)$, for $n$ odd, such that the action of $\pi_1(B,b)$ on $H^{2n}(\pi^{-1}(b);\QQ)$ is trivial. For classes $a, b \in H^*(E;\QQ)$  such that $A= \vert \pi_!(a) \vert$ and $B= \vert \pi_!(b) \vert$ are even, if $\pi_!(a)^k=0$ and $\pi_!(b)^l=0$, then $\pi_!(ab)^{(2g+1)(Ak + Bl)}=0$.
\end{theorem}

If $g \neq 1$, this is a consequence of~\cite[Lemma~5.7]{grigoriev-relations}. The proof we present here is independent of genus. 

\begin{proof}
First suppose that both $A$ and $B$ are at least $2$. Let $F$ be defined by the homotopy fibre sequence
$$F \longrightarrow K(\QQ,A) \times K(\QQ, B) \overset{\iota_A^k \times \iota_B^l}\longrightarrow K(\QQ,kA) \times K(\QQ, lB)$$
where the map between the Eilenberg--MacLane spaces is given by the product of the the appropriate powers of the fundamental cohomology classes. If $A$ is even it is an elementary calculation to see that 
$$H^*(F;\QQ) = \QQ[\iota_A, \iota_B]/(\iota_A^k, \iota_B^l),$$
and that $F$ is also simply-connected. Thus $F$ has $\QQ$-cohomological dimension $Ak+Bl-1$. 

As $\pi_!(a)^k=0$ and $\pi_!(b)^l=0$, the map
$$(\pi_!(a),\pi_!(b)) : B \longrightarrow K(\QQ,A) \times K(\QQ, B)$$
representing $(\pi_!(a), \pi_!(b))$ lifts to a map $f : B \to F$, and we let $B'$ denote the homotopy fibre of $f$, and $i : B' \to B$ the evaluation map. The class $i^*(\pi_!(ab))$ is the same as $\pi'_!(a' b')$, where $\pi' : E' \to B'$ is the pullback of $\pi$ along $i$, and $a'$ and $b'$ are the pullbacks of $a$ and $b$ along the map $E' \to E$ covering $i$. The classes $\pi'_!(a')$ and $\pi'_!(b')$ are trivial, so by~\cite[Theorem~2.6]{grigoriev-relations} we have $\pi'_!(a'b')^{2g+1}=0$. Thus $\pi_!(ab)^{2g+1}$ lies in the kernel of $i^*$, so has positive Serre filtration with respect to the map $f$. Thus $\pi_!(ab)^{(2g+1)(Ak + Bl)}$ has Serre filtration at least $Ak + Bl$, which is beyond the $\QQ$-cohomological dimension of $F$, so it is zero as required.

Finally, if $A=0$ and $B=0$ then this becomes \cite[Theorem 2.6]{grigoriev-relations}, which leaves the case $A=0$ and $B \geq 2$. In this case we must have $\pi_!(a)=0$, and the above argument can be followed using the homotopy fibre of $K(\QQ,B) \to K(\QQ,Bl)$ in the place of $F$.
\end{proof}

Let us also recall some further results from \cite{grigoriev-relations}, and a convenient relation between $R^*(W_g)$ and $R^*(W_g,\mpoint)$. We shall often write $\chi = \chi(W_g) = 2-2g$ for the Euler characteristic of $W_g$.

\begin{lemma}\label{lem:recollections}
Let $n$ be odd.
\begin{enumerate}[(i)]
\item\label{it:recollections:1} For any $c \in \mathcal{B}$, we have
\begin{equation}\label{eq:G1}
\chi^2 c - \chi\gkappa{ec} - \chi e \gkappa{c} + \gkappa{e^2} \gkappa{c} = 0 \in R^*(M_g,\mpoint)/\sqrt{0}.
\end{equation}
\item\label{it:recollections:2} We have
\begin{equation}\label{eq:G2}
(\chi-2)\chi e + \gkappa{e^2} = 0 \in R^*(M_g,\mpoint)/\sqrt{0}.
\end{equation}
\item\label{it:recollections:3} If $g \neq 1$ then the map
$$R^*(W_g) \longrightarrow R^*(W_g,\mpoint)$$
is injective.
\end{enumerate}
\end{lemma}
\begin{proof}
Item (\ref{it:recollections:1}) is Example 5.20 of \cite{grigoriev-relations}, but taking $a=\chi\nu_{(1*)} - e_{(*)}$ and $b = \chi p_{(*)} - e_{(*)}\kappa_p$ to avoid dividing by $\chi$. Item (\ref{it:recollections:2}) is Example 5.18 of \cite{grigoriev-relations}.

For item (\ref{it:recollections:3}), suppose that $x \in R^*(W_g)$ is a tautological class which vanishes in $R^*(W_g, \mpoint)$, and let $\pi : E \to B$ be a fibre bundle with fibre $W_g$. Then the pullback fibre bundle $\pi': \pi^*E \to E$ is a fibre bundle with fibre $W_g$ and canonical section, so $\pi^*(x(\pi))=0$. But then
$$(2-2g) \cdot x(\pi) = \pi_!(e(T_\pi) \cdot  \pi^*(x(\pi))) = 0$$
and $2-2g \neq 0$, so $x(\pi)=0$. This holds for any fibre bundle $\pi$, so $x=0 \in R^*(W_g)$.
\end{proof}

\begin{proposition}\label{prop:NilRelations}
Let $n$ be odd, $I = (i_1, i_2, \ldots, i_n)$ be a sequence, $p_I = p_1^{i_1} p_2^{i_2} \cdots p_n^{i_n}$ be the associated monomial in the Pontrjagin classes, and write $\vert I \vert = \sum_{j=1}^n i_j$. 
    \begin{enumerate}[(i)]
        \item\label{it:NilRelations:1} The class $\gkappa{p_I}$ is nilpotent in $R^*(W_g)$.
        \item\label{it:NilRelations:2} We have
        $$\chi^{\vert I \vert} \cdot \gkappa{ep_I} = \prod_{j=1}^n \gkappa{ep_j}^{i_j} \in R^*(W_g)/\sqrt{0}.$$
        
        \item\label{it:NilRelations:3} If $g > 1$ then the class $e$ is nilpotent in $R^*(W_g, \mpoint)$.
        
        \item\label{it:NilRelations:4} If $g \geq 1$ then for all $k > 1$ the class $\gkappa{e^k}$ is nilpotent in $R^*(W_g)$.
    \end{enumerate}
\end{proposition}

In Proposition \ref{prop:Genus1EulerTriv} we will show that $e \in R^*(W_1, \mpoint)$ is also nilpotent, though the proof is very different to that of this proposition.

\begin{proof}\mbox{}

(\ref{it:NilRelations:1}) By Theorem~\ref{thm:hirz-nilpotent}, the class $\gkappa{\Hirz_i}$ is nilpotent for all $i$. Therefore, for any monomial $\Hirz_J = \Hirz_1^{j_1} \Hirz_2^{j_2} \cdots \Hirz_n^{j_n}$, the class $\gkappa{\Hirz_J}$ is nilpotent by Theorem~\ref{thm:nilpotence-for-products}. As the $\Hirz_i$ and the $p_i$ generate the same subring of $H^*(BSO(2n);\QQ)$, it follows that any monomial $p_I$ may be written as a polynomial in the $\Hirz_i$, and hence that each class $\gkappa{p_I}$ is nilpotent.

(\ref{it:NilRelations:2}) Let $\pi : E \to B$ be a fibre bundle with fibre $W_g$, and for a monomial $p_J$ in Pontrjagin classes let us write
$$(p_J)' = \chi e p_J - e \pi^*(\gkappa{ep_J}).$$
We compute that
$$\kappa_{(p_J)' p_i} = \pi_!(\chi e p_i p_J - e p_i \pi^*(\gkappa{ep_J})) = \chi \kappa_{e p_i p_J} - \kappa_{ep_i} \kappa_{ep_J}$$
and that $\kappa_{(p_J)'}=0$ because $\pi_!(e) = \chi$.  By part (\ref{it:NilRelations:1}) the class $\kappa_{p_i}$ is nilpotent, so applying Theorem \ref{thm:nilpotence-for-products} it follows that $\kappa_{(p_J)' p_i}$ is nilpotent. So $\chi \kappa_{e p_i p_J} = \kappa_{ep_i} \kappa_{ep_J}$ modulo nilpotents. 
These identities hold for any such fibre bundle, so hold in $R^*(W_g)$, which by induction proves (\ref{it:NilRelations:2}).

(\ref{it:NilRelations:3}) By (\ref{it:NilRelations:1}), the class $\kappa_{e^2} = \kappa_{p_n}$ is nilpotent, so by \eqref{eq:G2} it follows that $e \in R^*(W_g, \mpoint)$ is also nilpotent as long as $g > 1$.

(\ref{it:NilRelations:4}) Suppose first that $g=1$, so that $\kappa_e=0$. Then as
$$\kappa_{e^{2l}} = \kappa_{p_n^l} \quad \text{ and } \quad \kappa_{e^{2l+1}} = \kappa_{e p_n^{l}}$$
it follows from (\ref{it:NilRelations:1}) that the former are nilpotent, and then from Theorem \ref{thm:nilpotence-for-products} that the latter are too.

For $g > 1$ the class $\kappa_e$ is not itself nilpotent. However, writing
$$\kappa_{e^{2l}} = \kappa_{p_n^l} \quad \text{ and } \quad \kappa_{e^{2l+1}} = \kappa_{e^3 p_n^{l-1}}$$
we see that the same argument will go through as soon as we show that $\kappa_{e^3}$ is nilpotent. 

Applying \eqref{eq:G1} with $c=p_n = e^2$ gives that
$$\chi^2 p_n - \chi\gkappa{ep_n} - \chi e \gkappa{p_n} + \gkappa{e^2} \gkappa{p_n} = 0 \in R^*(M_g,\mpoint)/\sqrt{0}.$$
The class $\kappa_{p_n}$ is nilpotent by (\ref{it:NilRelations:1}), and so we find that $\kappa_{ep_n} = \chi p_n \in R^*(W_g,\mpoint)/\sqrt{0}$. By (\ref{it:NilRelations:3}) the class $p_n = e^2$ is nilpotent, and therefore so is $\kappa_{ep_n} = \kappa_{e^3}$ in the ring $R^*(W_g,\mpoint)$. By Lemma~\ref{lem:recollections} (\ref{it:recollections:3})  the natural ring map $R^*(W_g) \to R^*(W_g,\mpoint)$ is injective, so $\kappa_{e^3}$ is nilpotent in $R^*(W_g)$ too, as required.
\end{proof}

\begin{corollary}\label{cor:Generation}
Let $n$ be odd.
\begin{enumerate}[(i)]
\item\label{it:Generation:1} $R^*(W_0)/\sqrt{0}$ is generated by $\kappa_{ep_1}, \kappa_{ep_2}, \ldots, \kappa_{ep_n}$,

\item\label{it:Generation:2} $R^*(W_1)/\sqrt{0} = \QQ$,

\item\label{it:Generation:3} $R^*(W_g)/\sqrt{0}$ is generated by  $\kappa_{ep_1}, \kappa_{ep_2}, \ldots, \kappa_{ep_{n-1}}$ for $g > 1$.
\end{enumerate}
\end{corollary}
\begin{proof}
It follows from Proposition \ref{prop:NilRelations} (\ref{it:NilRelations:1}) and (\ref{it:NilRelations:2}) that for any $g$ the ring $R^*(W_g)/\sqrt{0}$ is generated by the elements $\kappa_{ep_1}, \kappa_{ep_2}, \ldots, \kappa_{ep_n}$. 

For $g=1$ we have $\kappa_e=0$. It follows from Theorem \ref{thm:nilpotence-for-products} and Proposition \ref{prop:NilRelations} (\ref{it:NilRelations:1}) that the $\kappa_{ep_i}$ are nilpotent, and so $R^*(W_1)/\sqrt{0}=\QQ$. 

For $g > 1$ it follows from Proposition \ref{prop:NilRelations} (\ref{it:NilRelations:4}) that $\kappa_{ep_n}=\kappa_{e^3}$ is nilpotent. So $R^*(W_g)/\sqrt{0}$ is generated by the elements $\kappa_{ep_1}, \kappa_{ep_2}, \ldots, \kappa_{ep_{n-1}}$.
\end{proof}

\begin{corollary}\label{cor:Rewriting}
Let $n$ be odd and $c \in \mathcal{B}$, then $\chi c = \kappa_{ec} \in R^*(W_g,\mpoint)/\sqrt{0}$.
\end{corollary}
\begin{proof}
If $g = 1$ then $\chi c=0$, but also $\kappa_{ec}$ is nilpotent, by Corollary \ref{cor:Generation} (\ref{it:Generation:2}). If $g \neq 1$ then $\chi \neq 0$, and by \eqref{eq:G1} we have
$$\chi^2 c - \chi\gkappa{ec} - \chi e \gkappa{c} + \gkappa{e^2} \gkappa{c} = 0 \in R^*(W_g,\mpoint)/\sqrt{0}.$$
The class $\kappa_{e^2} = \kappa_{p_n}$ is nilpotent by Proposition \ref{prop:NilRelations} (\ref{it:NilRelations:1}), and the class $e$ is nilpotent by Proposition \ref{prop:NilRelations} (\ref{it:NilRelations:3}), this gives the required equation.
\end{proof}

\begin{proof}[Proof of Theorem \ref{thm:Main2} (\ref{it:Main2:1}) and  (\ref{it:Main2:2})]
Part (\ref{it:Main2:1}) is Corollary \ref{cor:Rewriting}. For part (\ref{it:Main2:2}), note that if $g \neq 1$ it follows that the map
$$R^*(W_g)/\sqrt{0} \longrightarrow R^*(W_g,\mpoint)/\sqrt{0}$$
is surjective, and under the same condition this map is injective by  Lemma~\ref{lem:recollections} (\ref{it:recollections:3}).
\end{proof}

In order to prove Theorem \ref{thm:Main1}, we must show that the generators of the rings $R^*(W_g)/\sqrt{0}$ given in Corollary \ref{cor:Generation} are algebraically independent. That is the subject of the next section.

\secone{Algebraic independence}\label{sec:Independence}

In this section, we finish the proof of Theorem \ref{thm:Main1} by showing the following.

\begin{theorem}\label{thm:algebraic-independence}
Let $n$ be either odd or even, and $\epsilon=1$ if $n$ is odd. Then
    \begin{enumerate}[(i)]
    \item\label{it:algebraic-independence:1} the map $\QQ[\kappa_{ep_1}, \kappa_{ep_2}, \ldots, \kappa_{ep_n}] \to R^*(W_0)/\sqrt{0}$ is injective,
    
    \item\label{it:algebraic-independence:2} the map $\QQ[\kappa_{ep_1}, \kappa_{ep_2}, \ldots, \kappa_{ep_{n-\epsilon}}] \to R^*(W_g)/\sqrt{0}$ is injective for $g > 1$.
    \end{enumerate}
\end{theorem}

When $n$ is odd, we have shown in Corollary \ref{cor:Generation} that these maps are both surjective, and so Theorem \ref{thm:Main1} follows immediately. We first explain the proof of this theorem in case (\ref{it:algebraic-independence:1}), as it is a simpler instance of the strategy we shall use to prove case (\ref{it:algebraic-independence:2}). Consider the fibre bundle
$$S^{2n} \longrightarrow BSO(2n) \overset{\pi}\longrightarrow BSO(2n+1),$$
which defines a ring homomorphism $R^*(W_0) \to H^*(BSO(2n+1);\QQ)$. The target is the ring $\QQ[p_1, p_2, \ldots, p_n]$ which contains no nilpotent elements, so the nilradical is in the kernel of this homomorphism. We may therefore consider the composition
$$\QQ[\kappa_{ep_1}, \kappa_{ep_2}, \ldots, \kappa_{ep_n}] \longrightarrow R^*(W_0)/\sqrt{0} \overset{\phi}\longrightarrow H^*(BSO(2n+1);\QQ).$$

The fibre bundle $\pi$ arises as the unit sphere bundle of the tautological bundle $\gamma_{2n+1} \to BSO(2n+1)$, so there is a bundle isomorphism $\epsilon^1 \oplus T_\pi \cong \pi^*\gamma_{2n+1}$, and so $p_i(T_\pi) = \pi^*(p_i)$. Thus we compute
$$\phi(\kappa_{ep_i}) = \pi_!(e(T_\pi) \cdot \pi^*(p_i)) = \pi_!(e(T_\pi)) \cdot p_i = 2 \cdot p_i$$
(as $\chi(S^{2n})=2$) so the composition is injective, which proves (\ref{it:algebraic-independence:1}).

\begin{remark}
This fibre bundle also shows that Proposition \ref{prop:NilRelations} (\ref{it:NilRelations:3}) and (\ref{it:NilRelations:4}) cannot be improved in the case $g=0$. Indeed, the vertical tangent bundle for the fibre bundle above is easily seen to be the tautological $2n$-dimensional bundle over $BSO(2n)$, so that $e(T_\pi ) = e \in H^{2n}(BSO(2n);\QQ)$, which is not nilpotent. Furthermore, $\kappa_{e^{2k+1}} = \kappa_{e p_n^k} = 2 \cdot p_n^k \in H^*(BSO(2n+1);\QQ)$ is not nilpotent either. 
\end{remark}

Our strategy for proving (\ref{it:algebraic-independence:2}) will be similar, and we begin by constructing a fibre bundle with fibre $W_g$ and base the classifying space of a Lie group.

  \sectwo{Constructing the Lie group action}
Let $G$ be the Lie group $SO(n) \times SO(n)$. An action of $G$ on the manifold $W_g$ gives rise to a smooth fibre bundle $E \to BG$ with fiber $W_g$. We will construct an example of such an action such that the characteristic classes $\gkappa{ep_i}$ with $1\leq i \leq d-\epsilon$ are algebraically independent in $H^*(BG;\QQ)$.  In our construction, $W_g$ will appear as a boundary of another $G$-invariant manifold.
  
  The standard $n$-dimensional representation of $SO(n)$ gives rise to two $n$-dimensional representations of $G=SO(n) \times SO(n)$.  For $i=1,2$, we let $V_i$ be the representation of $G$ where the $i$th copy of $SO(n)$ acts in the standard way and the other copy acts trivially. Let $\RR$ denote the trivial 1-dimensional representation. 
  
  \begin{proposition}\label{pro:hg-equiv-embedding}
      The $G$-representation $W = V_1 \oplus V_2 \oplus \RR$ contains an embedded compact smooth manifold with boundary $H_g \subset W$ enjoying the following properties:
      \begin{enumerate}[(i)]
          \item\label{it:hg-equiv-embedding:1} $\dim H_g = \dim W = 2n+1$,
          \item\label{it:hg-equiv-embedding:2} $H_g$ is $G$-invariant, so $\partial H_g$ is also $G$-invariant,
          \item\label{it:hg-equiv-embedding:3} $\partial H_g$ is a smooth manifold diffeomorphic to $W_g$. 
      \end{enumerate}
  \end{proposition}
    
  \begin{proof}
  We first illustrate the idea behind the construction of $H_g$ by presenting a construction that compromises on the smoothness conditions, but is correct up to homeomorphism.  
      
Let $B(V_1, 1)$ and $B(V_2 \oplus \RR, 1)$ denote the unit balls in the respective representations.  For any $g \in \NN$ and a sufficiently small $\varepsilon > 0$, it is possible to embed $g$ balls of radius $\varepsilon$ inside $B(V_2 \oplus \RR, 1)$ as disjoint and $SO(n)$-invariant subspaces. We can thus define the following subspaces, both invariant under the $G$-action ($X_g$ is pictured on Figure~\ref{fig:xg}).
  \begin{align*}
      X_g &= B(V_2 \oplus \RR, 1)\backslash \left( \bigsqcup_g B(V_2 \oplus \RR, \varepsilon)\right) \subset V_2 \oplus \RR\\
  H_g' &= B(V_1,1) \times X_g \subset V_1 \oplus V_2 \oplus \RR.
  \end{align*}
  The manifold $H_g' \subset W$ has codimension 0 and is $G$-invariant. Moreover, $\partial H_g'$ is by definition the boundary of a $(2n+1)$-manifold obtained by attaching $g$ trivial unlinked $n$-handles to $D^{2n+1}$, so is homeomorphic to $W_g$.

      \begin{figure}[tbh]
\begin{minipage}{\dimexpr0.5\textwidth-2\tabcolsep}
              \centering
              \begin{tikzpicture}
                  \filldraw[fill=black!30!white,thick, even odd rule] (0,0) circle (1.4) (0,0) circle (0.3) (0, 0.8) circle (0.3) (0, -0.8) circle (0.3);
                  \draw[dashed, very thick] (0,1.5) -- (0,-1.55);
              \end{tikzpicture}
 
              \captionsetup{font=footnotesize}
							\captionof{figure}{A cross-section of $X_g \subset \RR^3$ for $(n,g)=(2,3)$. $SO(2)$ acts by rotating about the axis.}
							\label{fig:xg}
          \end{minipage}
\begin{minipage}{\dimexpr0.5\textwidth-2\tabcolsep}
              \centering%
              \begin{tikzpicture}
                  \filldraw[fill=black!10!white,thick, pattern=dots] (0,2) -- (0,0) -- (2,0);
              \end{tikzpicture}\hspace{1.0cm}%
              \begin{tikzpicture}
                  \filldraw[fill=black!10!white,thick, pattern=dots] (0,2) -- (0,1.5) ..  controls (0,0.3) and (0.7,0) ..  (1.5,0) -- (2,0);
              \end{tikzpicture}
              \vspace{.3cm}
            
              \captionsetup{font=footnotesize}
							\captionof{figure}{The corner $C$ and the rounded corner $S$.}
							\label{fig:corner-smoothing}
          \end{minipage}          
      \end{figure}
  
  The construction so far does not prove the proposition, as $\partial H_g'$ is not a smooth submanifold of $W$. We would like to smooth out the corners of the manifold $H_g'$ to obtain a manifold $H_g$ that is again $G$-invariant, but has smooth boundary. The boundary of $H_g'$ is not smooth at the set $\partial B(V_1,1) \times \partial X_g \subset \partial H_g$. 

The Riemannian metrics on $B(V_1,1)$ and $X_g$ induced by those of $V_1$ and $V_2 \oplus \RR$ are $SO(n)$-invariant. Choosing inwards-pointing, nowhere-vanishing, $SO(n)$-invariant vector fields on $\partial B(V_1,1)$ and $\partial X_g$ (which may be achieved by choosing an inwards-pointing, nowhere-vanishing vector field, and then averaging over $SO(n)$) and integrating them, we may find a $G$-invariant neighbourhood
$$U' = \partial B(V_1,1) \times \partial X_g  \times C \hookrightarrow V_1 \oplus V_2 \oplus \RR$$
where $C = \left\{ (x,y) \in \RR^2 \mid x,y \geq 0 \right\}$ and has the trivial $G$-action.  
  
  Let $S \subset C$ be a strictly convex subset of $C$ that agrees with $C$ in a complement of a compact set around the origin, but has a smooth boundary diffeomorphic to $\RR$ (see Figure~\ref{fig:corner-smoothing}). We replace $U'$ with its ($G$-invariant) subset $U=\partial B(V_1,1) \times \partial X_g  \times S$, to obtain $H_g$. It clearly satisfies conditions (\ref{it:hg-equiv-embedding:1}) and (\ref{it:hg-equiv-embedding:2}) of Proposition~\ref{pro:hg-equiv-embedding}. Finally, $\partial H_g$ is by definition the boundary of a smooth $(2n+1)$-manifold obtained by attaching $g$ trivial unlinked $n$-handles to $D^{2n+1}$ and smoothing corners, so is diffeomorphic to $W_g$.
\end{proof}

\sectwo{Proof of Theorem \ref{thm:algebraic-independence} (\ref{it:algebraic-independence:2})}

Let us write
$$W_g \longrightarrow E \overset{\pi}\longrightarrow BSO(n) \times BSO(n)$$
for the fibre bundle of Proposition \ref{pro:hg-equiv-embedding}, which, as $H^*(BSO(n) \times BSO(n);\QQ)$ has no nilpotent elements, gives a ring homomorphism
$$\phi : R^*(W_g)/\sqrt{0} \longrightarrow H^*(BSO(n) \times BSO(n);\QQ).$$
Theorem \ref{thm:algebraic-independence} (\ref{it:algebraic-independence:2}) will follow if we show that $\phi(\kappa_{ep_1}), \phi(\kappa_{ep_2}), \ldots, \phi(\kappa_{ep_{n-\epsilon}})$ are algebraically independent.

By Proposition \ref{pro:hg-equiv-embedding}, the fibre bundle $\pi$ arises as a codimension 1 subbundle of the vector bundle $V_1 \oplus V_2 \oplus \RR$. We thus obtain a bundle isomorphism $\RR \oplus T_\pi \cong \pi^*(V_1 \oplus V_2 \oplus \RR)$, and hence $p_i(T_\pi) = \pi^*(p_i(V_1 \oplus V_2))$. Thus
$$\phi(\kappa_{ep_i}) = \pi_!(e(T_\pi)) \cdot p_i(V_1 \oplus V_2) = \chi \cdot p_i(V_1 \oplus V_2).$$
As $\chi \neq 0$, because we have assumed $g > 1$, the following lemma finishes the proof.

\begin{lemma}\label{lem:alg-ind-pi-bg}
The classes $p_i(V_1 \oplus V_2)$ for $1\leq i \leq n - \epsilon$ are algebraically independent in $H^*(BSO(n) \times BSO(n); \QQ)$.
\end{lemma}
\begin{proof}
Recall from~\cite[Theorem~15.21]{milnostash-characteristic} that in terms of the Pontrjagin and Euler classes of the tautological bundle over $BSO(d)$, we have
\[H^*(BSO(d); \QQ) = 
          \begin{cases}
              \QQ[p_1, \ldots, p_\frac{d-1}{2}] & d \text{ odd} \\
              \QQ[p_1, \ldots, p_{\frac{d}{2}-1}, e] & d \text{ even}.
\end{cases}\]
Moreover, if $d$ is even then $p_\frac{d}{2} = e^2$. In all cases, if $i > \frac{d-\epsilon}{2}$ then $p_i = 0$.
 
The block sum map $s: SO(n) \times SO(n) \to SO(2n)$ gives a map 
$$Bs^* : R= H^*(BSO(2n);\QQ)  \longrightarrow S = H^*(BSO(n) \times BSO(n);\QQ)$$
on cohomology, and the claim is that it is injective when restricted to the subalgbra $\QQ[p_1, p_2 \ldots, p_{n-\epsilon}]$. As the block sum map is faithful, it follows from a theorem of Venkov \cite{venkov} that $S$ is finitely-generated as a module over $R$. By the above: $S$ is a polynomial ring on $(n-1)$ generators if $n$ is odd, or $n$ generators if $n$ is even; $R$ is a polynomial algebra on $n$ generators.

We shall use the following lemma from commutative algebra. If $f : U \to V$ is a finite morphism between polynomial rings on the same (finite) number of generators, then it is injective. If it were not, then $\mathrm{Ker}(f)$ would be a proper prime ideal (because $V$, and hence $\mathrm{Im}(f)$, is an integral domain) so $U/\mathrm{Ker}(f)$ would have strictly smaller Krull dimension than $U$. But as $V$ is also finite over $U/\mathrm{Ker}(f)$ it too would have strictly smaller Krull dimension than $U$, a contradiction.

If $n$ is even then $Bs^* : R \to S$ is a finite morphism between polynomial rings on the same number of generators, so is injective. If $n$ is odd, note that $e \in H^{2n}(BSO(2n);\QQ)=R$ lies in the kernel of $Bs^*$ (as $V_1 \oplus V_2$ has trivial Euler class), so instead $R/(e) \to S$ is a finite morphism between polynomial rings on the same number of generators, so is injective. In either case, the subalgebra $\QQ[p_1, p_2 \ldots, p_{n-\epsilon}]$ of $H^*(BSO(2n);\QQ)$ injects under the block sum map. 
  \end{proof}

  \secone{Addenda}\label{sec:addenda}
  
    \sectwo{Classifying spaces}

We have defined the tautological ring $R^*(M)$ as a quotient of the abstract polynomial ring $\QQ[\kappa_c \, \vert \, c \in \mathcal{B}]$, but it may also be described as a subring of the cohomology of the classifying space $B\Diff^+(M)$ of the group of orientation-preserving diffeomorphisms of $M$.

Precisely, if $\pi : E \to B$ is equipped with the structure of a smooth oriented numerable fibre bundle with fibre $M^d$, though without assuming that $E$ or $B$ are themselves oriented compact smooth manifolds, then we still have  a vertical tangent bundle $T_\pi$ and a cohomological pushforward $\pi_!(-)$, so we may still define
$$\kappa_c(\pi) = \pi_!(c(T_\pi)) \in H^*(B;\mathbb{Q}).$$
The projection map
$$E\Diff^+(M) \times_{\Diff^+(M)} M \longrightarrow B\Diff^+(M)$$
is the universal example of such a fibre bundle, so there are universal classes $\kappa_c \in H^*(B\Diff^+(M);\mathbb{Q})$.

\begin{lemma}
The ring homomorphism $\QQ[\kappa_c \, \vert \, c \in \mathcal{B}] \to H^*(B\Diff^+(M);\QQ)$ has kernel $I_M$, and hence image isomorphic to $R^*(M)$.
\end{lemma}
\begin{proof}
The only point which needs discussion is that we defined $I_M$ in terms of the collection of all smooth fibre bundles over compact smooth oriented manifolds (in other words, proper submersions $\pi: E^{k+d} \to B^k$), and $B\Diff^+(M)$ is not one.

Let us write $I'_M$ for the kernel of this ring homomorphism. As $B\Diff^+(M)$ carries the universal smooth oriented fibre bundle with fibre $M$, any class in $I'_M$ must be trivial when evaluated on any proper submersion $\pi : E^{k+d} \to B^k$, so $I'_M \subset I_M$. The reverse inclusion holds because a rational cohomology class is zero if and only if it evaluates to zero on every rational homology class, and every rational homology class is represented by a map from a smooth (stably framed) manifold.
\end{proof}

The analogous argument identifies $R^*(M,\mpoint)$ with a subring of the cohomology of $B\Diff^+(M,\mpoint)$, the classifying space of the group of orientation preserving diffeomorphisms of $M$ which fix a point $\mpoint \in M$, and identifies $R^*(M,D^d)$ with a subring of the cohomology of $B\Diff^+(M,D^d)$, the classifying space of the group of orientation preserving diffeomorphisms of $M$ which fix a disc $D^d \subset M$
  
  \sectwo{Other classes of fibre bundles}
  
  We have defined the tautological ring $R^*(M) = \QQ[\kappa_c \, \vert \, c \in \mathcal{B}]/I_M$ in terms of the ideal $I_M$ of polynomials in the $\kappa_c$ which vanish when evaluated on every smooth oriented fibre bundle with fibre $M$. By varying this condition, we may form the following related tautological rings:
  \begin{enumerate}[(i)]
  \item $R^*_{\text{torelli}}(M)$, defined in terms of the ideal of polynomials in the $\kappa_c$ which vanish on every smooth oriented fibre bundle with fibre $M$ and homologically trivial action of the fundamental group of the base,
  
  \item $R^*_{0}(M)$, defined in terms of the ideal of polynomials in the $\kappa_c$ which vanish on every smooth oriented fibre bundle with fibre $M$ and isotopically trivial action of the fundamental group of the base.
  \end{enumerate} 
These are related by surjective ring homomorphisms
$$R^*(M)\longrightarrow R^*_{\text{torelli}}(M) \longrightarrow R^*_{0}(M).$$
  
  It follows from our results that for $M=W_g$ and $n$ odd, these maps all become isomorphisms after dividing out each ring by its nilradical: Sections \ref{sec:Hirz} and \ref{sec:Rels} give relations in $R^*(W_g)/\sqrt{0}$ whereas Section \ref{sec:Independence} shows algebraic independence of classes in $R^*_{0}(W_g)/\sqrt{0}$. In fact, this is true quite generally.
  
  \begin{proposition}\label{prop:ChangeOfClass}
  Let $M$ be a simply-connected manifold of dimension $d > 5$. Then
   $$R^*(M)/\sqrt{0}\longrightarrow R^*_{\text{torelli}}(M)/\sqrt{0} \longrightarrow R^*_{0}(M)/\sqrt{0}$$
   are isomorphisms. The same holds for $(M, \mpoint)$ and $(M, D^d)$.
  \end{proposition}
  \begin{proof}
  Both maps are epimorphisms by definition, so it is enough to show that the composition is an isomorphism. Sullivan has shown \cite[Theorem 13.3]{sullivan} that under the stated conditions the mapping class group  $\Gamma_M = \pi_0(\Diff^+(M))$ is commensurable to an arithmetic group. In particular, it has finite virtual cohomological dimension, and so finite $\QQ$-cohomological dimension.
  
  Letting $\Diff^+_0(M)$ denote the path component of the identity, there is a fibration sequence
  $$B\Diff^+_0(M) \overset{i}\longrightarrow B\Diff^+(M) \overset{p}\longrightarrow B\Gamma_M$$
  and an associated Serre spectral sequence. We then proceed much as in the proof of Theorem \ref{thm:nilpotence-for-products}: any class in 
  $$\mathrm{Ker}(i^* : H^*(B\Diff^+(M);\QQ) \to H^*(B\Diff^+_0(M);\QQ))$$
  has positive Serre filtration, so some power of it has Serre filtration beyond the $\QQ$-cohomological dimension of $B\Gamma_M$. Such a power is therefore zero, showing that $\mathrm{Ker}(i^*)$ consists of nilpotent elements, and hence that the map
  $$i^* : H^*(B\Diff^+(M);\QQ)/\sqrt{0} \longrightarrow H^*(B\Diff^+_0(M);\QQ)/\sqrt{0}$$
  is injective. The result now follows by the discussion of classifying spaces in the previous section.
  
  For $(M, \mpoint)$ and $(M, D^d)$ the argument is identical, but using the relative mapping class groups of these pairs. For $M$ simply-connected the natural map $\Gamma_{(M,\mpoint)} \to \Gamma_M$ is an isomorphism, as may be seen by the fibration sequence $M \to B\Diff^+(M, \mpoint) \to B\Diff^+(M)$. The fibration $SO(d) \to B\Diff^+(M, D^d) \to B\Diff^+(M,\mpoint)$ shows that $\Gamma_{(M, D^d)} \to \Gamma_{(M,\mpoint)}$ is onto with kernel of order at most 2. Thus $\Gamma_{(M, D^d)}$ still has finite $\QQ$-cohomological dimension.
  \end{proof}
  
  \sectwo{Genus zero}\label{sec:genus-zero}
  
  For the manifold $W_0 = S^{2n}$, we can improve Theorem \ref{thm:Main1} to say that
  $$R^*(W_0) = \QQ[\kappa_{ep_1}, \kappa_{ep_2}, \ldots, \kappa_{ep_n}],$$
  that is, we do not need to divide by the nilradical, nor need a condition on the parity of $n$. Let $\pi: E \to B$ be a fibre bundle with fibre $S^{2n}$. We use the following two strengthenings of our results.

    \begin{lemma}\label{lem:hirz-genus-0}
        $\gkappa{\Hirz_i}(\pi) = 0 \in H^*(B;\QQ)$.
    \end{lemma}
    \begin{proof}
    Observe that in the proof of Theorem \ref{thm:hirz-nilpotent}, the group $\mathrm{Aut}(H, \lambda)$ is trivial when $g=0$. Therefore, the classes $\gkappa{\Hirz_i}(\pi)$ must be pulled back through the cohomology of the contractible space $B\mathrm{Aut}(H, \lambda)$.
    \end{proof}

    \begin{lemma}\label{lem:PushGenusZero}
        If $p,q \in H^*(E;\QQ)$ satisfy $\pi_!(p)=0$ and $\pi_!(q) = 0$, then $\pi_!(pq)=0$. 
    \end{lemma}
    \begin{proof}
    It follows from the Gysin sequence for $\pi$ that there is a $p' \in H^*(B;\QQ)$ such that $p = \pi^*(p')$, but then  $\pi_!(pq)=\pi_!\left( \pi^*(p')q \right) = p'\cdot \pi_!(q)=0$.
    \end{proof}
    
It follows from Lemma \ref{lem:hirz-genus-0} that $\kappa_{\Hirz_i}=0 \in R^*(W_0)$, and so, by Lemma \ref{lem:PushGenusZero}, $\kappa_{\Hirz_J}=0 \in R^*(W_0)$ for any monomial $\Hirz_J$. As these give a basis for the subring of $H^*(BSO(2n);\QQ)$ generated by Pontrjagin classes, we have that $\kappa_{p_I}=0$ too. For a fibre bundle $\pi: E \to B$ with fibre $S^{2n}$, by the Gysin sequence we therefore have $p_I = \pi^*((p_I)')$ for some $(p_I)' \in H^*(B;\QQ)$, and so
$$\kappa_{ep_I}(\pi) = \pi_!(e(T_\pi) \pi^*((p_I)')) = \pi_!(e(T_\pi)) \cdot (p_I)' = 2(p_I)'.$$
In particular, we may express $\kappa_{ep_I}(\pi)$ by a universal formula in the $p_i' = \tfrac{1}{2}\kappa_{ep_i}(\pi)$, so $\kappa_{ep_1}, \kappa_{ep_2}, \ldots, \kappa_{ep_n}$ generate $R^*(W_0)$. They are algebraically independent in this ring as we have already show that they are algebraically independent in $R^*(W_0)/\sqrt{0}$.

    \sectwo{Genus one}\label{sec:genus-one}

We have two outstanding claims to prove in the case $g=1$. Firstly, Theorem \ref{thm:Main2} (\ref{it:Main2:3}), and secondly that $R^*(W_1,\mpoint)$ is finitely-generated, so that $R^*(W_1,D^{2n})$ is too and we can deduce Corollary \ref{cor:Main4} from Corollary \ref{cor:Main3}.

\secthree{The Euler class}

We first prove the analogue of Proposition \ref{prop:NilRelations} (\ref{it:NilRelations:3}) for $g=1$.

\begin{proposition}\label{prop:Genus1EulerTriv}
The class $e \in R^*(W_1, \mpoint)$ is nilpotent.
\end{proposition}

Our proof of this proposition will use the following, which would seem to be quite generally useful in the study of tautological rings.

\begin{theorem}\label{thm:EulerClassHtyInv}
Let $B$ be a finite CW-complex, let $\pi: E \to B$ and $\pi' : E' \to B$ be smooth fibre bundles with closed $d$-manifold fibres, and let $f : E \to E'$ be a map over $B$ which is a fibre homotopy equivalence. Then $f^*(e(T_{\pi'})) = e(T_\pi) \in H^d(E;\ZZ)$.
\end{theorem}
\begin{proof}
We refer to \cite{BG} for background on the following, and to \cite[Chapter 15]{MaySigurdsson} for technical details of fibrewise Spanier--Whitehead duality. For a space $X$ over $B$, let $X_+ = X \sqcup B$ be the associated ex-space, and $\Sigma^\infty_B X_+$ denote the fibrewise suspension spectrum. By \cite[Theorem 4.7]{BG} or \cite[Theorem 15.1.1]{MaySigurdsson} the parameterised spectrum $\Sigma^\infty_B E_+$ is dualisable, and in fact its dual can be made rather explicit. Choose a smooth embedding $E \subset B \times \RR^N$ with tubular neighbourhood $U$ and projection $p : U \to E$, and let $U^+_B$ denote the 1-point compactification of $U$ formed fibrewise over $B$, so there is a canonical section $s_\infty : B \to U^+_B$ given by the points at $\infty$. There is then a map
\begin{align*}
B \times S^N & \longrightarrow E_+ \wedge_B U^+_B\\
(b, x) & \longmapsto \begin{cases}
s(b) \wedge_B s_\infty(b) & (b,x) \not\in U\\
p(b,x) \wedge_B (b,x) & (b,x) \in U
\end{cases}
\end{align*}
giving a map 
$$\mu_E : \Sigma^\infty_B B_+ \longrightarrow \Sigma^\infty_B E_+ \wedge_B \Sigma^{\infty-N}_BU^+_B$$
of parameterised spectra which exhibits $\Sigma^\infty_BE_+$ and $\Sigma^{\infty-N}_BU^+_B$ as Spanier--Whitehead dual. There is therefore a complementary duality map
$$\hat{\mu}_E : \Sigma^{\infty-N}_BU^+_B \wedge_B \Sigma^\infty_B E_+ \longrightarrow \Sigma^\infty_B B_+$$
and the composition
\begin{align*}
\Sigma^\infty_B B_+ &\overset{\mu_E}\longrightarrow \Sigma^\infty_B E_+ \wedge_B \Sigma^{\infty-N}_BU^+_B \overset{\cong}\longrightarrow \Sigma^{\infty-N}_BU^+_B  \wedge_B \Sigma^\infty_B E_+\\
 &\overset{\mathrm{Id} \wedge_B \Delta}\longrightarrow \Sigma^{\infty-N}_BU^+_B  \wedge_B \Sigma^\infty_B E_+ \wedge_B \Sigma^\infty_B E_+\\
 & \overset{\hat{\mu}_E \wedge_B \mathrm{Id}}\longrightarrow \Sigma^{\infty}_B B_+  \wedge_B \Sigma^\infty_B E_+ \overset{\simeq}\longrightarrow \Sigma^\infty_B E_+ 
\end{align*}
is a lift of the Becker--Gottlieb transfer of $\pi$ to parameterised spectra. Dualising this map gives a map of parameterised spectra
$$\epsilon_E: \Sigma^{\infty-N}_BU^+_B \longrightarrow \Sigma^\infty_B B_+$$
and so, base changing along $B \to \{*\}$ and then collapsing $B$ to a point, we obtain a map of spectra
$$c: \Sigma^{\infty-N}U^+ \longrightarrow \Sigma^\infty B_+ \longrightarrow \Sigma^\infty S^0.$$
On the other hand, $U$ is homeomorphic to the normal bundle $\nu_\pi$ of $E$ in $B \times \RR^N$, so $U^+$ is homeomorphic to the Thom space $\mathrm{Th}(\nu_\pi)$. The pullback along $c$ of the canonical cohomology class in $H^0(\Sigma^\infty S^0;\ZZ)$, followed by the Thom isomorphism (as $T_\pi$, and hence $\nu_\pi$, is oriented), therefore gives a class in $H^{d}(E;\ZZ)$. It follows from the Poincar{\'e}--Hopf theorem that this is $e(T_\pi)$.

The analogous construction, using a tubular neighbourhood $E' \subset V \subset B \times \RR^N$, describes $e(T_{\pi'})$. The point is that the above construction used only fibrewise Spanier--Whitehead duality, and so if $D(f_+) : \Sigma^{\infty-N}_BV_B^+ \to \Sigma^{\infty-N}_BU_B^+$ is the dual of $f_+$ then we have 
$$\epsilon_E \circ D(f_+) \simeq \epsilon_{E'} : \Sigma^{\infty-N}_B V_B^+ \longrightarrow \Sigma^\infty_B B_+$$
and so, as under the Thom isomorphism $D(f_+)^*$ induces the map $(f^*)^{-1} : H^*(E;\ZZ) \to H^*(E';\ZZ)$, we have $f^*(e(T_{\pi'})) = e(T_\pi)$.
\end{proof}

The arguments of this proof can be refined to \emph{define} an Euler class of the vertical tangent bundle, $e(T_\pi) \in H^d(E;\ZZ)$, when $\pi : E \to B$ is a fibration whose fibre is a Poincar{\'e} duality space of dimension $d$. In particular, for such a Poincar{\'e} duality space $M^d$ one can define the universal such class $e \in H^d(B\mathrm{hAut}^+(M, \mpoint);\ZZ)$.

\begin{proof}[Proof of Proposition \ref{prop:Genus1EulerTriv}]
Let $\pi: E^{k+2n} \to B^k$ be a homologically trivial smooth oriented fibre bundle with fibre $W_1$, and let $s : B \to E$ be a section. We will show that $s^*(e(T_\pi))=0 \in H^{2n}(B;\QQ)$, so that $e=0 \in R^*_{\text{torelli}}(W_1, \mpoint)$. Hence, by Proposition \ref{prop:ChangeOfClass}, $e$ is nilpotent in $R^*(W_1, \mpoint)$.

Under the stated assumptions the Serre spectral sequence for $\pi$ has a product structure and collapses, giving classes $a,b \in H^n(E;\ZZ)$ which restrict to a basis of $H^n(W_1;\ZZ)$. The obstructions to finding a lift
\begin{equation*}
\xymatrix{
 & S^n \times S^n \ar[d]\\
E \ar[r]_-{(a,b)} \ar@{-->}[ur] & K(\ZZ,n) \times K(\ZZ,n)
}
\end{equation*}
lie in $H^{i+1}(E; \pi_i(S^n \times S^n))$ for $n < i < k+2n$, which are all torsion groups.

Let $ \cdots \to B_2 \overset{f_1}\to B_1 \overset{f_0}\to B_0 = B$ be a tower of $k$-dimensional CW-complexes in which each $f_i^* : H^*(B_i;\ZZ) \to H^*(B_{i+1};\ZZ)$ annihilates all torsion but is rationally injective (such a tower may be formed by iteratedly taking a $k$-skeleton of the homotopy fibre of the canonical map $B_i \to \prod_{j=1}^k K(\mathrm{tors} H^j(B_i;\ZZ),j)$). We may pull back $\pi$ and $s$ to obtain homologically trivial bundles $\pi_i : E_i \to B_i$ with sections $s_i$. Each of these also have Serre spectral sequences which have a product structure and collapse, and they are connected by maps $\hat{f}_i : E_{i+1} \to E_i$ covering the $f_i$.

It follows that $\hat{f}_i^*$ always sends a torsion class to one of higher Serre filtration, and so in particular (as the Serre spectral sequence for each $\pi_i$ has three rows) that the composition of three such maps annihilates any torsion class. Therefore a finite composition of such maps annihilates all obstructions for finding the dashed lift above, so for some $i$ we have a map $t : E_i \to S^n \times S^n$ splitting the inclusion of the fibre. It follows that $\pi_i$ is fibre homotopy equivalent to the trivial $W_1$-bundle, and hence by Theorem \ref{thm:EulerClassHtyInv} we have
$$e(T_{\pi_i})=0 \in H^n(E_i;\ZZ)$$
and so $s_i^*e(T_{\pi_i}) = 0 \in H^n(B_i;\ZZ)$. The maps $f_i^*$ were rationally injective, and so $s^*e(T_\pi)=0 \in H^{2n}(B;\QQ)$, as required.
\end{proof}

Finally, we can prove the third part of Theorem \ref{thm:Main2}.

\begin{proof}[Proof of Theorem \ref{thm:Main2} (\ref{it:Main2:3})]
We have shown that $e$ is nilpotent in $R^*(W_1, \mpoint)$. Consider the fibre bundle
  $$S^n \times S^n \longrightarrow BSO(n) \times BSO(n) \overset{\pi}\longrightarrow BSO(n+1) \times BSO(n+1)$$
and let $\pi' : E \to BSO(n) \times BSO(n)$ be the fibre bundle obtained by pulling $\pi$ back along itself. Then $\pi'$ has a section $s'$, and the ring $H^*(BSO(n) \times BSO(n);\QQ)$ has no nilpotent elements, so there is a ring homomorphism
  $$\phi : R^*(W_1, \mpoint)/\sqrt{0} \longrightarrow H^*(BSO(n) \times BSO(n);\QQ).$$
There is a bundle isomorphism $(s')^*(T_{\pi'}) \cong V_1 \oplus V_2$, so $\phi(p_i) = p_i(V_1 \oplus V_2)$. It follows from Lemma \ref{lem:alg-ind-pi-bg} that $p_1, p_2, \ldots, p_{n-1} \subset R^*(W_1, \mpoint)/\sqrt{0}$ are algebraically independent.
\end{proof}

\secthree{Finite generation}

The proof in \cite{grigoriev-relations} showing that $R^*(W_g)$ is finitely generated does not apply for $g \leq 1$. We have computed $R^*(W_0)$ completely, and can observe that it is finitely generated, which leaves only the case $g=1$. In fact, we do not know whether this ring is finitely generated. However, we have the following.
  
  \begin{proposition}\label{prop:fg}
      For all $g\geq 0$, the ring $R^*(W_g,\mpoint)$ is finitely generated. 
  \end{proposition}

  \begin{proof}
      Our argument follows~\cite[Example~5.20 and Lemma~5.21]{grigoriev-relations} and we shall freely use the language and notation of that paper. We work in the space $\mathcal{M}_g(\{1,2,\mpoint\})$ with three marked points $1, 2,$ and $\mpoint$. Let $p \in H^*( BSO(2n);\QQ)$ be even-dimensional, and define
      \begin{align*}
      a & =\gpsi{p}{\mpoint} - \intclass{1\mpoint} \kappa_p \in H^*(\mathcal{M}_g(\{1,2,\mpoint\});\QQ)\\
      b & =\intclass{1\mpoint} - \intclass{2\mpoint} \in H^*(\mathcal{M}_g(\{1,2,\mpoint\});\QQ).
      \end{align*}
 Both of these classes push forward to zero in $H^*(\mathcal{M}_g(\{1,2\});\QQ)$. According to~\cite[Theorem~2.6]{grigoriev-relations} we therefore have
\begin{multline*} \label{eq:relation-fin-gen-mp}
    0 = \left( \vphantom{\sum^{1}}  \left(\pi^{\{1,2, \mpoint\}}_{\{1, 2\}} \right)_! \left( \vphantom{\sum} \left(\gpsi{p}{\mpoint}- \intclass{1\mpoint} \gkappa{p}\right)\left(\intclass{1\mpoint} - \intclass{2\mpoint} \right) \right) \right) ^{2g+1} = \\
    = \left( \gpsi{p}{1} - \gpsi{p}{2} - \gpsi{e}{1} \gkappa{p} + \intclass{12}\gkappa{p} \right)^{2g+1} \in H^*(\mathcal{M}_g(\{1,2\});\QQ).
\end{multline*}

For any $q \in H^*(BSO(2n);\QQ)$ the above relation can be multiplied by $\gpsi{q}{2}$ and pushed forward to $H^*(\mathcal{M}_g(1);\QQ))$ to obtain a relation that expresses $\gkappa{p^{2g+1}q}$ in terms of decomposable elements in the tautological subring $R^*(W_g, \mpoint) \subset H^*(\mathcal{M}_g(1);\QQ)$. (Recall that this ring is generated by the $\kappa_c$ as well as $c$ for $c \in H^*( BSO(2n);\QQ )$.) As explained in~\cite[Proof of Theorem~1.1]{grigoriev-relations}, this is sufficient to show that the ring is finitely generated.
  \end{proof}
  
  It follows that the ring $R^*(W_g,D^{2n})$ is also finitely generated.

\bibliographystyle{halpha} 
\bibliography{bib}

\end{document}